\documentclass{amsart}

\begin{document}
\renewcommand\theenumi{(\roman{enumi})}
\renewcommand\labelenumi{{\theenumi}}

 % \articletype{Research Article}

%  \author*[1,2]{Mancho Manev}
  \author[M. Manev]{Mancho Manev}
%  \author[2]{...}
%  \author[1]{...}
%  \runningauthor{M. Manev}
  \address[1]{Department of Algebra and Geometry, Faculty of Mathematics and Informatics,
Plovdiv University, 236 Bulgaria Blvd, Plovdiv, 4027, Bulgaria, e-mail: mmanev@uni-plovdiv.bg}
  \address[2]{Department of Medical Informatics, Biostatistics and Electronic Educa\-ti\-on,
Faculty of Public Health, Medical University of Plovdiv, 15A Vasil Aprilov Blvd, Plovdiv, 4002, Bulgaria, e-mail: mmanev@meduniversity-plovdiv.bg}
  \title[Natural Connections with Totally Skew-Sym\-met\-ric Torsion \dots]
  {Natural Connections with Totally Skew-Sym\-met\-ric Torsion on Manifolds with
Almost Contact 3-Structure and Metrics of Hermitian-Norden Type}
%  \runningtitle{Natural Connections with Totally Skew-Sym\-met\-ric Torsion \dots
  %Manifolds with Almost Contact 3-Structure and Metrics of Hermitian-Norden Type
%  }
%  \subtitle{...}
\begin{abstract}
It is considered a differentiable manifold equipped with a pseudo-Riemannian metric and an almost contact 3-struc\-ture so that an almost contact metric structure and two almost contact B-metric structures
  are generated.
There are introduced the so-called associated Nijenhuis tensors for the studied structures.
It is given a geometric interpretation of the vanishing of these
tensors as a necessary and sufficient condition for the existence of
linear connections with totally skew-symmetric torsions preserving the
structure. An example of a 7-dimensional manifold with connections of the considered
type is given.
\end{abstract}
  \keywords{Almost contact manifold, 3-structure, B-metric, hypercomplex structure, Hermitian metric, Norden metric, associated Nijenhuis tensor, natural connection, totally skew-symmetric torsion.}
\subjclass[2010]{Primary 53C05, 53C15; Secondary 53C50,  53C26}
%  \communicated{...}
%  \dedication{...}
%  \received{...}
%  \accepted{...}

%  \journalyear{...}
%  \journalvolume{..}
%  \journalissue{..}
%  \aop
%  \DOI{...}

\newcommand{\MM}{\mathcal{M}}
\newcommand{\R}{\mathbb{R}}
\newcommand{\X}{\mathfrak{X}}
\newcommand{\f}{\varphi}
\newcommand{\ep}{\varepsilon}
\newcommand{\ea}{\varepsilon_\alpha}
\newcommand{\eb}{\varepsilon_\beta}
\newcommand{\eg}{\varepsilon_\gamma}
\newcommand{\n}{\nabla}
\newcommand{\W}{\mathcal{W}}
\newcommand{\al}{\alpha}
\newcommand{\bt}{\beta}
\newcommand{\gm}{\gamma}
\newcommand{\lm}{\lambda}
\newcommand{\ta}{\theta}
\newcommand{\fa}{\varphi_{\alpha}}
\newcommand{\xia}{\xi_{\alpha}}
\newcommand{\etaa}{\eta_{\alpha}}
\newcommand{\om}{\omega}
\newcommand{\D}{\mathrm{d}}
\newcommand{\s}{\mathfrak{S}}
\newcommand{\sx}{\mathop{\mathfrak{S}}\limits_{x,y,z}}
\newcommand{\LL}{\mathcal{L}}
\newcommand{\LLL}{\mathfrak{L}}
\newcommand{\SSS}{\mathcal{S}}
\newcommand{\C}{\mathbb{C}}
\newcommand{\F}{\mathcal{F}}
\newcommand{\ddt}{\textstyle\frac{\D}{\D t}}
\newcommand{\Ja}{J_{\alpha}}
\newcommand{\G}{\mathcal{G}}
\newcommand{\V}{\mathcal{V}}
\newcommand{\pd}{\partial}
\newcommand{\ddx}{\frac{\pd}{\pd x^i}}
\newcommand{\ddy}{\frac{\pd}{\pd y^i}}
\newcommand{\ddu}{\frac{\pd}{\pd u^i}}
\newcommand{\ddv}{\frac{\pd}{\pd v^i}}
\newcommand{\dda}{\frac{\pd}{\pd a}}
\newcommand{\ddb}{\frac{\pd}{\pd b}}
\newcommand{\ddc}{\frac{\pd}{\pd c}}
\newcommand{\Lfa}{(L,\allowbreak\f_{\al},\allowbreak\xi_{\al},\allowbreak\eta_{\al},g)}

\newtheorem{defn}{Definition}
\newtheorem{conv}{Convention}
\newtheorem{rmk}{Remark}
\newtheorem{thm}{Theorem}
\newtheorem{lem}[thm]{Lemma}
\newtheorem{prop}[thm]{Proposition}
\newtheorem{cor}[thm]{Corollary}

\newcommand{\thmref}[1]{The\-o\-rem~\ref{#1}}
\newcommand{\propref}[1]{Pro\-po\-si\-ti\-on~\ref{#1}}

\maketitle

\section{Introduction}\label{intro}

It is known the notion of an \emph{almost contact 3-structure} on a differentiable manifold of dimension $4n+3$ (\cite{Kuo,Udr}). The product of a man\-i\-fold with almost contact 3-structure and a real line admits an \emph{almost hypercomplex structure} (cf. \cite{Kuo,AlMa}).

It is only considered the case of equipping of such a manifold with a Riemannian metric compatible with each of the three structures in the given almost contact 3-structure. This is the so-called \emph{almost contact metric 3-structure}.

In \cite{Man51}, we have introduced a pseudo-Riemannian metric which has another kind of compatibility with the triad of almost contact structures on a manifold with almost contact 3-structure.
The product of this manifold of new type and a real line is a $(4n+4)$-dimensional manifold which admits an almost hypercomplex structure $(J_1,J_2,J_3)$ and a Hermitian-Norden metric (briefly, an HN-metric), i.e. $J_1$ (resp., $J_2$ and $J_3$) acts as an
isometry (resp., act as anti-isometries) with respect to the pseudo-Riemannian metric of neutral signature in each tangent fibre.
This structure is called an \emph{almost hypercomplex HN-metric structure} and it is studied in \cite{GriManDim12,GriMan24,Man28,ManGri32}, etc. The constructed structure on $(4n+3)$-dimensional manifolds we call an \emph{almost contact 3-structure with metrics of Hermitian-Norden type} (briefly, an HN-type).

The goal of the present paper is to introduce an appropriate tensor on a manifold with almost contact 3-structure and metrics of HN-type such that the vanishing of this tensor is a necessary and sufficient condition for existence of linear connections with totally skew-symmetric torsion preserving the almost contact 3-structure and the metric of HN-type.

%\paragraph*{Convention} Let $\MM$ be an almost contact manifold and $\MM\times\R$ be the corresponding almost complex manifold.

\begin{conv}
Let $\MM$ be an almost contact manifold and $\MM\times\R$ be the corresponding almost complex manifold.
\begin{enumerate}
  \item We shall use $x$, $y$, $z$, $\dots$ to denote smooth vector fields on $\MM$, i.e. $x, y, z\in \X(\MM)$, or vectors in the tangent space $T_p\MM$ at $p\in \MM$;
  \item We shall use $X$, $Y$, $Z$, $\dots$ to denote smooth vector fields on $\MM\times\R$ or tangent vectors in $T_{\bar{p}}(\MM\times\R)$ at $\bar{p}\in \MM\times\R$.
\end{enumerate}
\end{conv}

\section{Manifolds with almost contact HN-metric 3-structure}\label{sec:1}

Let $(\MM,\f_{\al},\xi_{\al},\eta_{\al})$, $(\al=1,2,3)$ be a manifold
with an almost contact 3-structure, % or an \emph{almost hypercontact manifold},
i.e. $\MM$
is a $(4n+3)$-dimensional differentiable manifold with three almost
contact structures $(\f_{\al},\xi_{\al},\eta_{\al})$, $(\al=1,2,3)$ consisting of endomorphisms
$\f_{\al}$ of the tangent bundle,  Reeb vector fields $\xi_{\al}$ and their dual contact 1-forms
$\eta_{\al}$ satisfying the following identities:
\begin{equation}\label{str}
\begin{array}{c}
\f_{\al}\circ\f_{\bt} = -\delta_{\al\bt}I + \xi_{\al}\otimes\eta_{\bt}+\epsilon_{\al\bt\gm}\f_{\gm},\\
\f_{\al}\xi_{\bt} = \epsilon_{\al\bt\gm}\xi_{\gm},\quad
\eta_{\al}\circ\f_{\bt}=\epsilon_{\al\bt\gm}\eta_{\gm},\quad \eta_{\al}(\xi_{\bt})=\delta_{\al\bt},
\end{array}
\end{equation}
where $\al,\bt,\gm\in\{1,2,3\}$, $I$ is the identity on the algebra $\X(\MM)$% on the smooth vector fields on $\MM$
, $\delta_{\al\bt}$ is the Kronecker delta, $\epsilon_{\al\bt\gm}$ is the Levi-Civita symbol, i.e. either the sign of the permutation $(\al,\bt,\gm)$ of $(1,2,3)$ or 0 if any index is repeated.

% for all cyclic permutations $(\al, \bt, \gm)$ of $(1,2,3)$.

%Since any compatible metric and any B-metric on an almost contact manifold $\MM$ are metrics corresponding to a Hermitian metric and a Norden metric on the corresponding almost complex manifold $\MM\times\R$ (or on the corresponding contact distribution $H=\ker(\eta)$), respectively, we said that the compatible metric and the B-metric are metrics of Hermitian type and Norden type on $\MM$, respectively.

In \cite{Man51}, we introduce the following notion.
%\begin{defn}[]
%Let $\MM$ be equipped with an almost contact 3-structure $(\f_{\al},\xi_{\al},\eta_{\al})$, $\al=1,2,3$.
A pseudo-Riemannian metric $g$ is called a \emph{metric of Hermitian-Norden type} (in short an \emph{HN-metric}) on a manifold with almost contact 3-structure $(\MM,\allowbreak{}\f_{\al},\allowbreak{}\xi_{\al},\allowbreak{}\eta_{\al})$, if it satisfies the identities
\begin{equation}\label{HN-met}
  g(\f_{\al}x,\f_{\al}y)=\ea g(x,y)+\eta_{\al}(x)\eta_{\al}(y),\quad \al=1,2,3
\end{equation}
for some cyclic permutation $(\ep_1,\ep_2,\ep_3)$ of $(1,-1,-1)$.
We suppose for the sake of definiteness that
\[%\label{ea}%
 \ea=
\begin{cases}
\begin{array}{ll}
1, \quad & \al=1;\\
-1, \quad & \al=2,3.
\end{array}
\end{cases}
\]
Then, $(\f_{\al},\xi_{\al},\allowbreak{}\eta_{\al},g)$ we call an \emph{almost contact HN-metric 3-structure}.
% or an \emph{almost hypercontact HN-metric structure}.
%\end{defn}

Actually, this 3-structure consists of an almost contact metric structure for $\al=1$ and two almost contact B-metric structures for $\al=2$ and $\al=3$. Then $g$ is a compatible metric with respect to $(\f_{1},\xi_{1},\eta_{1},g)$ and $g$ is a B-metric with respect to $(\f_{2},\xi_{2},\eta_{2},g)$ and $(\f_{3},\xi_{3},\eta_{3},g)$ on $\MM$.

%Bearing in mind the structural groups of the almost contact 3-structures with compatible metric (\cite{Kuo}) and %the hypercomplex manifolds with HN-metric structure (\cite{GriMan24}), we can conclude the following.
%The structural group of the manifolds with almost contact HN-metric 3-structure is $(GL(n,\mathbb{H})\cap O(2n,2n))\times O{(2,1)}$, where $GL(n,\mathbb{H})$ is  the group of invertible quaternionic $(n\times n)$-matrices and $O(p,q)$ is the pseudo-orthogonal group of signature $(p,q)$ for natural numbers $p$ and $q$.
%$H^{\bot}=\mathrm{span}{\xi_1,\xi_2,\xi_3}$.

The fundamental tensors of a manifold with almost contact HN-metric 3-struc\-ture are the three
$(0,3)$-tensors determined by
\begin{equation}\label{F}
F_\al (x,y,z)=g\bigl( \left( \n_x \f_\al
\right)y,z\bigr),\qquad
\al=1,2,3,
\end{equation}
where $\n$ is the Levi-Civita connection generated by $g$.
They have the following basic properties caused by the structures
\begin{equation}\label{Fa-prop}
\begin{array}{l}
  F_{\al}(x,y,z)=-\ea F_{\al}(x,z,y)\\
  \phantom{F_{\al}(x,y,z)}
  =-\ea F_{\al}(x,\f_{\al}y,\f_{\al}z)+F_{\al}(x,\xi_{\al},z)\,
  \eta_{\al}(y)\\
  \phantom{F_{\al}(x,y,z)=-\ea F_{\al}(x,\f_{\al}y,\f_{\al}z)}
    +F_{\al}(x,y,\xi_{\al})\,\eta_{\al}(z).
\end{array}
\end{equation}

Bearing in mind the following consequence of \eqref{HN-met}
\begin{equation}\label{eta-g}
\eta_{\al}=-\ea \xi_{\al} \lrcorner\, g,
\end{equation}
as well as \eqref{F}, we have the following relations
\begin{equation}\label{Fetaxia}
  F_\al(x,\f_\al y,\xi_\al)=-\ea \left(\n_x\eta_\al \right)(y)=g\left(\n_x\xi_\al,y\right).
\end{equation}

Let $\LLL_{\xi_\al}g$ denote the Lie derivative of $g$ along $\xi_\al$. We have the following relations using \eqref{HN-met}
\begin{equation}\label{Lie-der}
\begin{array}{l}
  (\LLL_{\xi_\al}g)(x,y)=g(\n_x \xi_\al,y)+g(x,\n_y \xi_\al)\\
  \phantom{(\LLL_{\xi_\al}g)(x,y)}
  =-\ea \bigl(\left(\n_x\eta_\al\right)(y)+\left(\n_y\eta_\al\right)(x)\bigr).
\end{array}
\end{equation}

%The following associated 1-forms, defined as traces of $F_{\al}$, are known as their \emph{Lee forms}:
%\begin{equation}\label{ta}
%\begin{array}{c}
%\theta_{\al}(z)=g^{ij}F_{\al}(e_i,e_j,z),\quad \theta^*_{\al}(z)=g^{ij}F_{\al}(e_i,\f_{\al}e_j,z),\\ \om_{\al}(z)=F_{\al}(\xi_{\al},\xi_{\al},z),
%\end{array}
%\end{equation}
%where $g^{ij}$ are the components of the inverse matrix of the metric $g$ with respect to an arbitrary basis of the type $\{e_1,e_2,\dots,e_{4n+2},\xi_{\al}\}$.

We use the known classifications of the almost contact metric manifolds and the almost contact B-metric manifolds in terms of $F_\al$ given in \cite{AlGa} and \cite{GaMiGr}, respectively. The former classification is relevant for $\al=1$ and contains 12 basic classes $\W_i$ $(i=1,2,\dots,12)$, whereas  the latter one consists of 11 basic classes $\F_i$ $(i=1,2,\dots,11)$ and it applies for $\al=2$ or $\al=3$.

%The simplest case  of the manifolds with
%almost contact HN-metric 3-structure is when the
%structures are $\n$-parallel, $\n\f_\al=\n\xi_\al=\n\eta_\al=\n g=\n
%\widetilde{g}_\al=0$, and it is determined by the condition
%$F_\al=0$.
%We call these manifolds \emph{cosymplectic HN-metric manifolds} or \emph{manifolds with cosymplectic HN-metric 3-structure}.

\section{Associated Nijenhuis tensors on manifolds with almost contact HN-metric 3-structure}

As it is known, for each $\al\in\{1,2,3\}$ the Nijenhuis tensor $N_{\al}$ of an almost contact manifold $(\MM,\f_\al,\allowbreak{}\xi_\al,\eta_\al)$ is defined by:
\[
  N_{\al}=[\f_{\al},\f_{\al}]+\xi_{\al}\otimes\D\eta_{\al},
\]
where we have $[\f_{\al},\f_{\al}](x,y)=\f_{\al}^2[x,y]+[\f_{\al}x,\f_{\al}y]-\f_{\al}[\f_{\al}x,y]-\f_{\al}[x,\f_{\al}y]$ and $\D\eta_{\al}(x,y)=\left(\n_x\eta_{\al}\right)y-\left(\n_y\eta_{\al}\right)x$.
Moreover, let us recall, if two almost contact structures in an almost contact 3-structure are normal, then the third one is also normal \cite{Kuo,YaAk}.

%\subsection{Associated Nijenhuis tensor on $(\MM,\f_1,\xi_1,\allowbreak{}\eta_1,g)$}

%We give the following %Besides the {Nijenhuis} tensor $N_1$ for a $(\f_1,\xi_1,\allowbreak{}\eta_1)$-structure
Let us consider the symmetric braces $\{x,y \}$ introduced by the following equalities for a pseudo-Riemannian metric $g$
\begin{equation}\label{braces}
\begin{array}{l}
g(\{x,y\},z)=g(\nabla_xy+\nabla_yx,z)\\
\phantom{g(\{x,y\},z)}
=xg(y,z)+yg(x,z)-zg(x,y)-g([y,z],x)+g([z,x],y).
\end{array}
\end{equation}

%\begin{defn}
For the almost contact structure $(\f_1,\xi_1,\allowbreak{}\eta_1)$ and the metric $g$, we define a symmetric tensor $\widehat N_1$ by
\begin{equation}\label{hatN1}
\widehat N_1=\{\f_1,\f_1\}-\xi_1\otimes\LLL_{\xi_1}g,
\end{equation}
where %$\LL$ denotes the Lie derivative, and
%$\LLL_{\xi_1}$ denotes the Lie derivative along $\xi_1$ and
$\{\f_1 ,\f_1\}$
is the symmetric tensor field of type $(1,2)$ given by
\begin{equation}\label{{f1f1}}
\{\f_1 ,\f_1\}(x,y)=\{\f_1 x,\f_1 y\}+(\f_1)^2\{x,y\}-\f_1\{\f_1 x,y\}-\f_1\{x,\f_1
y\}.
\end{equation}
We call $\widehat N_1$ an  \emph{as\-so\-ci\-ated {Nijenhuis} tensor} on $(\MM,\f_1,\xi_1,\allowbreak{}\eta_1,g)$.
%\end{defn}

The corresponding tensors of type $(0,3)$ for $N_1$ and $\widehat N_1$ are given
by $N_1(x,y,z)=g(N_1(x,y),z)$ and $\widehat N_1(x,y,z)=g(\widehat
N_1(x,y),z)$, respectively.

By direct consequences of the definitions, we get that
$N_1$, $\widehat N_1$ and $\LLL_{\xi_1}g$ are expressed in
terms of $F_1$ as follows:
\begin{gather}
\begin{array}{l}
N_1(x,y,z)=
F_1(\f_1 x,y,z)+ F_1(x,y,\f_1 z)+F_1(x,\f_1 y,\xi_1)\,\eta_1(z)
\label{N1=F1}\\
\phantom{N_1(x,y,z)}
-F_1(\f_1 y,x,z)- F_1(y,x,\f_1 z)-F_1(y,\f_1x,\xi_1)\,\eta_1(z),
\end{array}\\
\begin{array}{l}
\widehat N_1(x,y,z)=
F_1(\f_1 x,y,z)+ F_1(x,y,\f_1 z)+F_1(x,\f_1 y,\xi_1)\,\eta_1(z)\\
\phantom{\widehat N_1(x,y,z)}
+F_1(\f_1 y,x,z)+ F_1(y,x,\f_1z)+F_1(y,\f_1x,\xi_1)\,\eta_1(z),\label{N1hat=F1}
\end{array}\\
\left(\LLL_{\xi_1}g\right)(x,y)=F_1(x,\f_1y,\xi_1)+F_1(y,\f_1x,\xi_1).
\label{Lxi1g=F1}
\end{gather}

%\subsection{Associated Nijenhuis tensor on $(\MM,\f_2,\xi_2,\eta_2,g)$}

%
In \cite{ManIv36}, it is defined the \emph{as\-so\-ci\-ated
Nijenhuis} tensor $\widehat N_2$ for the almost contact B-metric structure $(\f_2,\xi_2,\allowbreak{}\eta_2)$
by
\begin{equation}\label{hatN2}
\widehat N_2=\{\f_2,\f_2\}+\xi_2\otimes\LLL_{\xi_2}g,
\end{equation}
where %$\LL$ denotes the Lie derivative, and
$\{\f_2 ,\f_2\}$
is the symmetric tensor field of type $(1,2)$ defined as in \eqref{{f1f1}}.

\begin{prop}\label{prop:hatN=0=>Kill}
  For the almost contact B-metric manifold $(\MM,\f_2,\xi_2,\eta_2,g)$, the vanishing of $\widehat N_2$ implies that $\xi_2$ is Killing.
\end{prop}
\begin{proof}
%For $(\MM,\f_2,\xi_2,\eta_2,g)$,
It is known the formula for $F_2$ in terms of $N_2$ and $\widehat N_2$ from \cite{IvManMan45}, whereas the expression of $\widehat N_2$ by $F_2$ is given in \cite{ManIv36}:
\[%\label{F=NhatN}
\begin{split}
F_2(x,y,z)&=-\frac14\bigl\{N_2(\f_2 x,y,z)+N_2(\f_2 x,z,y)\\
&\phantom{=-\frac14\bigl\{ }+\widehat N_2(\f_2 x,y,z)+\widehat N_2(\f_2 x,z,y)\bigr\}\\
&\phantom{=\ }+\frac12\eta_2(x)\bigl\{N_2(\xi_2,y,\f_2 z)+\widehat
N_2(\xi_2,y,\f_2 z)\\
&\phantom{=\ \,+\frac12\eta_2(x)\bigl\{N_2(\xi_2,y,\f_2 z)}
+\eta_2(z)\widehat N_2(\xi_2,\xi_2,\f_2 y)\bigr\},
\\
\widehat N_2(x,y,z)&=
F_2(\f_2 x,y,z)- F_2(x,y,\f_2 z)+F_2(x,\f_2 y,\xi_2)\,\eta_2(z)\\
&%\phantom{\widehat N_2(x,y,z)=}
\, +F_2(\f_2 y,x,z)- F_2(y,x,\f_2z)+F_2(y,\f_2x,\xi_2)\,\eta_2(z).
\end{split}
\]
By the latter equalities, \eqref{Fetaxia} and \eqref{Lie-der}, we obtain the following relation
\begin{equation*}\label{Lxig=hatN}
\begin{split}
(\LLL_{\xi_2}g)(x,y)&=-\frac12\bigl\{\widehat N_2(\f_2 x,\f_2 y,\xi_2)+\widehat N_2(\xi_2,\f_2 x,\f_2 y)
+\widehat N_2(\xi_2,\f_2 y,\f_2 x)\bigr\}\\
&\phantom{=-\frac12\bigl\{ }
+\eta_2(x)\widehat N_2(\xi_2,\xi_2,y)+\eta_2(y)\widehat N_2(\xi_2,\xi_2,x)\bigr\},
\end{split}
\end{equation*}
which yields the statement.
\end{proof}

Let us remark that a similar statement of \propref{prop:hatN=0=>Kill} for an almost contact metric manifold is not true.

%Now, we shall discuss on the integrability and related topics when
Let the manifold $\MM$ be equipped with an almost contact 3-structure $(\f_{\al},\xi_{\al},\allowbreak{}\eta_{\al})$, $(\al=1,2,3)$
and then we consider the product $\MM\times\R$. % of %an almost contact manifold $\MM$ with a real line.
Let $X$ be a vector field on $\MM\times\R$ which is presented by a pair $\left(  x, a\ddt\right)$, where $x$ is a tangent vector field on $\MM$, $t$ is the coordinate on $\R$ and $a$ is a differentiable function on $\MM\times\R$ \cite[Sect. 6.1]{Blair}.
The almost complex structures $J_{\al}$, $(\al=1,2,3)$ are defined on the manifold $\MM\times\R$ by
\begin{equation}\label{JaX}
\begin{array}{c}
J_{\al}X=J_{\al}\left(x, a\ddt\right)=\left(
  \f_{\al}x-a\xi_{\al},
  \eta_{\al}(x)\ddt
\right).
\end{array}
\end{equation}
In such a way, in \cite{YaIsKo}, it is defined an almost hypercomplex structure on $\MM\times\R$ when $\MM$ has an almost contact 3-structure.

Moreover, we equip $\MM\times\R$ with the product metric $G=g-\D t^2$. By virtue of \eqref{JaX}, \eqref{HN-met} and its consequence $g(\xia,\xia)=-\ea$, we obtain
\[
G\bigl(\Ja(x,a\ddt),\Ja(y,b\ddt)\bigr)=\ea G\bigl((x,a\ddt),(y,b\ddt)\bigr),
\]
i.e. the manifold $\MM\times\R$ has an almost hypercomplex HN-metric structure $(\Ja,G)$, $(\al=1,2,3)$.

%For $X=\left(x, a\ddt\right)$, $Y=\left(y, b\ddt\right)$ in $\X(\MM\times\R)$ their Lie bracket has the form
We introduce the braces $\{X,Y\}$ for the vector fields $X=\bigl(x,a\ddt\bigr)$ and $Y=\bigl(y,b\ddt\bigr)$ on $\MM\times\R$ defined by
\begin{equation}\label{sym-skobi}
\begin{array}{c}
\{X,Y\}=\bigl(\{x,y\},(x(b)+y(a))\ddt\bigr),
\end{array}
\end{equation}
where $\{x,y\}$ are given in \eqref{braces}. Obviously, the braces are symmetric. %, i.e. $\{X,Y\}=\{Y,X\}$.

It is known from \cite{KoNo}, the Nijenhuis tensor of two endomorphisms $J_{\al}$ and $J_{\bt}$ has the following form:
\[
\begin{split}
2[J_{\al},J_{\bt}](X,Y)&=[J_{\al}X,J_{\bt}Y]-J_{\al}[J_{\bt}X,Y]-J_{\al}[X,J_{\bt}Y]
\\
&%\phantom{=}
+[J_{\bt}X,J_{\al}Y]-J_{\bt}[J_{\al}X,Y]-J_{\bt}[X,J_{\al}Y]\\
&+(J_{\al}J_{\bt}+J_{\bt}J_{\al})[X,Y].
\end{split}
\]
Moreover, the Nijenhuis tensor of an almost complex structure $J_{\al}\equiv J_{\bt}$ is presented by
\[
[J_{\al},J_{\al}](X,Y)=[J_{\al}X,J_{\al}Y]-J_{\al}[J_{\al}X,Y]
-J_{\al}[X,J_{\al}Y]-[X,Y].
\]

Analogously of the last two equalities, using the braces \eqref{sym-skobi} instead of the Lie brackets, we define consequently the associated Nijenhuis tensors in the two respective cases as follows:
\[
\begin{split}
2\{J_{\al},J_{\bt}\}(X,Y)&=\{J_{\al}X,J_{\bt}Y\}-J_{\al}\{J_{\bt}X,Y\}-J_{\al}\{X,J_{\bt}Y\}\\
&%\phantom{2\{J_{\al},J_{\bt}\}(X,Y)=}
+\{J_{\bt}X,J_{\al}Y\}-J_{\bt}\{J_{\al}X,Y\}-J_{\bt}\{X,J_{\al}Y\}\\
&%\phantom{2\{J_{\al},J_{\bt}\}(X,Y)=}
+(J_{\al}J_{\bt}+J_{\bt}J_{\al})\{X,Y\},\nonumber%\label{JaJb}
\end{split}
\]
\begin{equation}\label{JaJa}
\begin{array}{l}
\{J_{\al},J_{\al}\}(X,Y)
=\{J_{\al}X,J_{\al}Y\}-J_{\al}\{J_{\al}X,Y\}-J_{\al}\{X,J_{\al}Y\}\\
\phantom{\{J_{\al},J_{\bt}\}(X,Y)=}
-\{X,Y\}.
\end{array}
\end{equation}
The latter tensor is given in \cite{Man48} and coincides with the tensor $\tilde N$ introduced in \cite{GaBo} by an equivalent equality of \eqref{JaJa}.

According to \cite{GrHe}, the $\G_1$-manifolds are almost Hermitian manifolds whose corresponding Nijenhuis (0,3)-tensor by the Hermitian metric is a 3-form.
This condition is equivalent to the vanishing of the associated Nijenhuis tensor, according to \cite{Man52}.

As it is known from \cite{GaBo}, the class %$\W_3$
of the quasi-K\"ahler manifolds with Norden metric is the only basic class of the considered manifolds with non-integrable almost complex structure $J$, because $[J,J]$ is non-zero there. Moreover, this class is determined by the condition $\{J,J\} = 0$.

In \cite{Man52}, it is proven %an analogous statement  of \propref{prop:NJ}
%for the associated Nijenhuis tensors, i.e. it is valid
the following
\begin{prop}\label{prop:hatNJ}
Let $(J_1,J_2,J_3)$ be an almost hypercomplex structure and $G$ is a pseudo-Riemannian metric on the almost hypercomplex manifold. If two of its six associated Nijenhuis tensors
$\{J_1,J_1\}$, $\{J_2,J_2\}$, $\{J_3,J_3\}$, $\{J_1,J_2\}$, $\{J_1,J_3\}$, $\{J_2,J_3\}$
vanish, then the others also vanish.
\end{prop}

We seek to express in terms of the structure tensors of $(\f_{\al},\xi_{\al},\eta_{\al})$ a necessary and sufficient condition for $\{J_{\al},J_{\al}\}=0$.

For the structure $(\fa,\xia,\etaa,g)$, $\al\in\{1,2,3\}$, let us define the following four tensors of type (1,2), (0,2), (1,1), (0,1), respectively:
\begin{equation}\label{hatN1234}
\begin{array}{l}
\widehat N_{\al}^{(1)}(x,y)=
\{\f_{\al},\f_{\al}\}(x,y)-\ea\left(\LLL_{\xi_{\al}}g\right)(x,y)\cdot\xi_{\al},\\
\widehat N_{\al}^{(2)}(x,y)=
-\ea\left(\LLL_{\xi_{\al}}g\right)(\f_{\al}x,y)
-\ea\left(\LLL_{\xi_{\al}}g\right)(x,\f_{\al}y),
%=-\ea\bigl(\left(\LLL_{\xi_{\al}}g\right)\barwedge\f_{\al}\bigr)(x,y),
\\
\widehat N_{\al}^{(3)}x=
\{\f_{\al},\f_{\al}\}(\fa x,\xia)+\left(\LLL_{\xi_{\al}}\eta_{\al}\right)(\f_{\al}x)\cdot\xi_{\al}
%\phantom{\widehat N_{\al}^{(3)}x=}
+2\eta_{\al}(x)\f_{\al}\n_{\xi_{\al}}\xi_{\al},\\
\widehat N_{\al}^{(4)}(x)=
-\left(\LLL_{\xi_{\al}}\eta_{\al}\right)(x).
\end{array}
\end{equation}
\begin{prop}\label{prop:JaJa=0_hatN1234=0}
The associated Nijenhuis tensor $\{J_{\al},J_{\al}\}$ of an almost complex structure $\Ja$ for some $(\MM\times\R,\Ja,G)$, $\al\in\{1,2,3\}$, vanishes
  if and only if the four tensors $\widehat N_{\al}^{(1)}$, $\widehat N_{\al}^{(2)}$, $\widehat N_{\al}^{(3)}$, $\widehat N_{\al}^{(4)}$ for the structure $(\fa,\xia,\etaa,g)$ vanish.
\end{prop}

\begin{proof}%[Proof of \propref{prop:JaJa=0_hatN1234=0}]
First of all we need of the following relations
%\begin{lem}
\begin{equation}\label{LgLetaom}
\left(\LLL_{\xi_{\al}}g\right)(\xi_{\al},x)=-\ea\left(\LLL_{\xi_{\al}}\eta_{\al}\right)(x)
%=-\ea\om_{\al}(\f_{\al} x)
=g\left(\n_{\xi_{\al}} \xi_{\al},x\right).
\end{equation}
%\end{lem}
%\begin{proof}
  These equalities follow by virtue of %\eqref{str},
  \eqref{eta-g}, \eqref{Fetaxia}, \eqref{Lie-der}.
%\end{proof}

Since any $\{J_{\al}, J_{\al}\}$ is a tensor field
of type $(1, 2)$, it suffices to compute the tensors $\{J_{\al}, J_{\al}\}((x, 0), (y, 0))$ and $\{J_{\al}, J_{\al}\}((x, 0), (o, \ddt))$, where $o$ is the zero element of $\X(\MM)$. Taking into account \eqref{{f1f1}}, \eqref{JaX}, \eqref{sym-skobi}, \eqref{JaJa}, we obtain consequently:
\begin{align*}
\{J_{\al},J_{\al}\}&\bigl((x,0\ddt),(y,0\ddt)\bigr)=\\
=&\
    \bigl\{\left(\fa x,\etaa(x)\ddt\right), \left(\fa y,\etaa(y)\ddt\right)\bigr\}
    -\bigl\{\left(x,0\ddt\right), \left(y,0\ddt\right)\bigr\}\\
 &-\Ja\bigl\{\left(\fa x,\etaa(x)\ddt\right), \left(y,0\ddt\right)\bigr\}
    -\Ja\bigl\{\left(x,0\ddt\right), \left(\fa y,\etaa(y)\ddt\right)\bigr\}\\
=&\ \bigl(\left\{\fa x,\fa y\right\},\ \left(\fa x(\etaa(y))+\fa y(\etaa(x))\right)\ddt\bigr)\\
&-\bigl(-\fa^2\left\{x,y\right\}+\etaa\left(\{x,y\}\right)\xia,\ 0\ddt\bigr)\\
&-\bigl(\fa\{\fa x,y\}-y\left(\etaa(x)\right)\xia,\ \etaa\left(\{\fa x,y\}\right)\ddt\bigr)\\
&-\bigl(\fa\{x,\fa y\}-x\left(\etaa(y)\right)\xia,\ \etaa\left(\{x,\fa y\}\right)\ddt\bigr)\\
=&\ \bigl(\widehat N_{\al}^{(1)}(x,y),\ \widehat N_{\al}^{(2)}(x,y)\ddt\bigr),\\
\{J_{\al},J_{\al}\}&\bigl((x,0\ddt),(o,\ddt)\bigr)=\\
=&\
 \bigl\{\left(\fa x,\etaa(x)\ddt\right), \left(-\xia,0\ddt\right)\bigr\}
   -\bigl\{\left(x,0\ddt\right), \left(o,\ddt\right)\bigr\}\\
 &-\Ja\bigl\{\left(\fa x,\etaa(x)\ddt\right), \left(o,\ddt\right)\bigr\}
   -\Ja\bigl\{\left(x,0\ddt\right), \left(-\xia,0\ddt\right)\bigr\}\\
=& -\bigl(\{\fa x,\xia\}, \xia\left(\etaa(x)\right)\ddt\bigr)
    + \bigl(\fa\{x,\xia\}, \etaa\left(\{x,\xia\}\right)\ddt\bigr)\\
=&\ \bigl(\widehat N_{\al}^{(3)}x,\ \widehat N_{\al}^{(4)}(x)\ddt\bigr).%\tag*{\qed}
\end{align*}
Then, for any $\al=1,2,3$, the vanishing of $\{J_{\al},J_{\al}\}$ holds if and only if  $\widehat N_{\al}^{(1)}$, $\widehat N_{\al}^{(2)}$, $\widehat N_{\al}^{(3)}$, $\widehat N_{\al}^{(4)}$ vanish.
\end{proof}

\begin{prop}\label{prop:N1=0=>N234=0}
For an almost contact structure $(\f_{\al},\xi_{\al},\eta_{\al})$, $\al\in\{1,2,3\}$ and a pseudo-Riemannian metric $g$, the vanishing of
$\widehat N_{\al}^{(1)}$ implies the vanishing of $\widehat N_{\al}^{(2)}$, $\widehat N_{\al}^{(3)}$ and $\widehat N_{\al}^{(4)}$.
\end{prop}
\begin{proof}
We set $y=\xia$ in $\widehat N_{\al}^{(1)}(x,y)=0$ and apply $\etaa$. Then, using \eqref{{f1f1}} and  \eqref{str}, we obtain $\left(\LLL_{\xi_{\al}}g\right)(x,\xi_{\al})=0$ and thus $\widehat N_{\al}^{(4)}=0$, according to \eqref{LgLetaom}.

Therefore, %Bearing in mind
from the form of $\widehat N_{\al}^{(1)}$ in \eqref{hatN1234},
we get
$
\{\f_{\al},\f_{\al}\}(\f_{\al}x,\xi_{\al})=0.
$
On the other hand, bearing in mind \eqref{LgLetaom}, we have that the vanishing of $\left(\LLL_{\xi_{\al}}g\right)(x,\xi_{\al})$ is equivalent to the vanishing of $\left(\LLL_{\xi_{\al}}\eta_{\al}\right)(x)$ and $\n_{\xi_{\al}}\xi_{\al}$. Thus,
we obtain $\widehat N_{\al}^{(3)}=0$.

Finally, applying $\etaa$ to $\widehat N_{\al}^{(1)}(\fa x,y)=0$
and using \eqref{{f1f1}}, we have
\[
\eta_{\al}\bigl(\{\f_{\al}^2x,\f_{\al}y\}\bigr)-\ea\left(\LLL_{\xi_{\al}}g\right)(\f_{\al}x,y)=0.
\]
The first term in the latter equality can be expressed in the following form \[-\ea\left(\LLL_{\xi_{\al}}g\right)(x,\f_{\al}y),\] 
using that $\LLL_{\xi_{\al}}\eta_{\al}$ vanishes. In such a way we obtain that $\widehat N_{\al}^{(2)}(x,y)=0$.
\end{proof}

%Obviously, we have
\begin{prop}
For an almost contact structure $(\f_{\al},\xi_{\al},\eta_{\al})$, $\al\in\{1,2,3\}$ and a pseudo-Riemannian metric $g$, where  $\xia$ is Killing, $\widehat N_{\al}^{(2)}$ and $\widehat N_{\al}^{(4)}$ vanish.
Moreover, we have the following:
\begin{enumerate}
  \item $\widehat N_{\al}^{(1)}$ vanishes if and only if $\{\f_{\al},\f_{\al}\}$ vanishes;
  \item $\widehat N_{\al}^{(3)}$ vanishes if and only if %the condition %$\{\fa x,\xia\}=\fa\{x,\xia\}$
  $\xia\lrcorner\{\fa,\fa\}$  vanishes.
\end{enumerate}
\end{prop}
\begin{proof}
Taking into account that $\LLL_{\xi_{\al}}g$ vanishes, we have $\widehat N_{\al}^{(1)}=
\{\f_{\al},\f_{\al}\}$ and $\widehat N_{\al}^{(2)}=0$. Further, we obtain $\widehat N_{\al}^{(4)}=0$ and $\widehat N_{\al}^{(3)}x=
\{\f_{\al},\f_{\al}\}(\fa x,\xia)$, according to \eqref{LgLetaom}. Then, (i) is obvious whereas (ii) holds, bearing in mind the assumption for $\xia$.
\end{proof}

%\begin{defn}
Let $(\MM,\f_{\al},\xi_{\al},\eta_{\al},g)$, $(\al=1,2,3)$ be a manifold with almost contact HN-metric 3-structure. The symmetric $(1,2)$-tensors defined by
\begin{equation}\label{Nhat-3}
\widehat N_{\al}=\{\f_{\al},\f_{\al}\}-\ea\xi_{\al}\otimes \LLL_{\xi_{\al}}g
\end{equation}
we call \emph{associated Nijenhuis tensors} on $(\MM,\f_{\al},\xi_{\al},\eta_{\al},g)$.
%\end{defn}

The corresponding $(0,3)$-tensors are denoted by
\[
\widehat N_{\al}(x,y,z)=g(\widehat N_{\al}(x,y),z),\quad
\{\f_{\al},\f_{\al}\}(x,y,z)=g(\{\f_{\al},\f_{\al}\}(x,y),z).
\]
Then, taking into account \eqref{eta-g} and \eqref{Nhat-3}, we obtain
\[
\widehat N_{\al}(x,y,z)=\{\f_{\al},\f_{\al}\}(x,y,z)+\left(\LLL_{\xi_{\al}}g\right)(x,y)\,\etaa(z).
\]

%In such a way, we obtain the following
\begin{thm}\label{thm:Ja=0_hatN=0}
Let $(\MM,\f_{\al},\xi_{\al},\eta_{\al},g)$, $(\al=1,2,3)$ be a manifold
with almost contact HN-metric 3-structure.
For any $\al$, the associated Nijenhuis tensor
$\{J_{\al},J_{\al}\}$ of the almost complex structure $J_{\al}$  on $(\MM\times\R,\Ja,G)$ vanishes if and only if the associated Nijenhuis tensor $\widehat{N}_{\al}$ %on $\MM$
of the structure $(\f_{\al},\xi_{\al},\eta_{\al},g)$ vanishes. % and $\xi_{\al}$ is a Killing vector field.???
\end{thm}

\begin{proof}
  The statement follows from \propref{prop:JaJa=0_hatN1234=0} and  \propref{prop:N1=0=>N234=0}, bearing in mind \eqref{hatN1234} and \eqref{Nhat-3}.
\end{proof}

%Then, we establish the truthfulness of the following
\begin{thm}\label{thm:hatN}
Let $(\MM,\f_{\al},\xi_{\al},\eta_{\al},g)$, $(\al=1,2,3)$ be a manifold
with almost contact HN-metric 3-structure.
If two of the associated Nijenhuis tensors $\widehat N_{\al}$ vanish, the third also vanishes.
%: %for the associated Nijenhuis tensors of an almost contact 3-struc\-ture %and different indices $\al,\bt,\gm\in\{1,2,3\}$,
%we have:
%  \begin{enumerate}
%    \item $\widehat N_1$ vanishes;% and $\xi_1$ is Killing;
%    \item $\widehat N_2$ vanishes;
%    \item $\widehat N_3$ vanishes.
%  \end{enumerate}
\end{thm}
\begin{proof}
It follows by virtue of %\propref{cor:Nxi-al},
\propref{prop:hatNJ} and \thmref{thm:Ja=0_hatN=0}.
\end{proof}

\section{Natural connections with totally skew-symmetric torsion}

A linear connection $D$ is said to be a \emph{natural connection for $(\f_\al,\xi_\al,\eta_\al,g)$},  $\al\in\{1,2,3\}$, if it preserves the structure, i.e.
\[
D\f_\al=D\xi_\al=D\eta_\al=Dg=0.
\]

\begin{thm}\label{thm:NN=Nhat}
Let $(\MM,\f_1,\xi_1,\eta_1,g)$ be a pseudo-Riemannian manifold with an almost contact metric structure.
%For the {Nijenhuis} tensor and the associated {Nijenhuis} tensor of .
The following statements are equivalent:
\begin{enumerate}
  \item The manifold belongs to the class $\W_2\oplus\W_4\oplus\W_9\oplus\W_{10}\oplus\W_{11}$ determined by
\begin{equation}\label{F1_W3W7}
F_1 (\f_1 x,y,z)+F_1(\f_1y,x,z)+F_1(x,y,\f_1z)+F_1(y,x,\f_1z)=0.
\end{equation}
  \item The associated {Nijenhuis} tensor $\widehat N_1$ vanishes and $\xi_1$ is a Killing vector field;
  \item The tensor $\{\f_1,\f_1\}$ vanishes and $\xi_1$ is a Killing vector field;
  \item The {Nijenhuis} tensor $N_1$ is a 3-form and $\xi_1$ is a Killing vector field;
  \item There exists a natural connection $D^1$ with totally skew-symmetric torsion for the structure $(\f_1,\xi_1,\eta_1,g)$ and this connection is unique.
\end{enumerate}
\end{thm}

\begin{proof}
Using \eqref{N1hat=F1} and \eqref{Lxi1g=F1}, we have that the vanishing of $\widehat N_1$ and $\LLL_{\xi_1}g$ implies the identity \eqref{F1_W3W7}. Viceversa, setting $x=y=\xi_1$ in \eqref{F1_W3W7}, we have $F_1(\xi_1,\xi_1,z)=0$. If we put $x=\f_1x$, $y=\f_1y$, $z=\xi_1$ in \eqref{F1_W3W7} and use the latter vanishing, we obtain that $\LLL_{\xi_1}g=0$ and therefore $\widehat N_1=0$. The determination of the class in (i) by \eqref{F1_W3W7} follows immediately from the definition of the basic classes of the classification in \cite{AlGa}. So, the equivalence between (i) and (ii) is valid.

Now, we need to prove the following relation
\begin{equation}\label{NN=Nhat}
\widehat N_1(x,y,z)=N_1(z,x,y)+N_1(z,y,x).
\end{equation}
We calculate the right hand side of \eqref{NN=Nhat} using \eqref{N1=F1}. Taking into account \eqref{Fa-prop} and their consequence
\begin{equation}\label{FffF1}
\begin{array}{l}
    F_1(x,y,\f_1z)=F_1(x,\f_1y,z)+F_1(x,\xi_1,\f_1y)\eta_1(z)\\
    \phantom{F_1(x,y,\f_1z)=F_1(x,\f_1y,z)}
    +F_1(x,\xi_1,\f_1z)\eta_1(y),
\end{array}
\end{equation}
we obtain
\[
\begin{array}{l}
N_1(z,x,y)+N_1(z,y,x)=
-F_1(\f_1x,z,y)-F_1(\f_1y,z,x)\\
\phantom{N_1(z,x,y)+N_1(z,y,x)=}
-F_1(x,z,\f_1y)-F_1(y,z,\f_1x)\\
\phantom{N_1(z,x,y)+N_1(z,y,x)=}
-F_1(x,\f_1z,\xi_1)\eta_1(y)-F_1(y,\f_1z,\xi_1)\eta_1(x).
\end{array}
\]
Using again \eqref{FffF1} and the first equality in \eqref{Fa-prop}, we establish that the right hand side of the latter equality is equal to $\widehat N_1(x,y,z)$, according to \eqref{N1hat=F1}.
Therefore, \eqref{NN=Nhat} is valid.

The relation \eqref{NN=Nhat} implies the equivalence between (ii) and (iv), whereas
the equivalence between (iv) and (v) is given in Theorem 8.2 of \cite{FrIv02}.
The equivalence between (ii) and (iii) follows from \eqref{hatN1}.
\end{proof}

For the natural connection $D^1$ with totally skew-symmetric torsion for the structure $(\f_1,\xi_1,\eta_1,g)$, we have
\begin{equation}\label{D1T1}
g\left(D^1_{x} y,z\right)=g(\n_x y, z)+\frac12 T_1(x,y,z)
\end{equation}
and its torsion $T_1$, according to Theorem 8.2 of \cite{FrIv02}, is determined in our notations by
\begin{equation}\label{T1Kuch}
\begin{array}{l}
T_1=-\eta_1\wedge\D\eta_1+\D^\f_1 \Phi+N_1-\eta_1\wedge(\xi_1 \lrcorner N_1),\\
\end{array}
\end{equation}
where it is used the notation $\D^{\f_1} \Phi(x,y,z)=-\D \Phi(\f_1 x,\f_1 y,\f_1 z)$ for the fundamental 2-form
$\Phi$ of the almost contact metric structure, i.e. $\Phi(x,y)=g(x,\f_{1}y)$.

Since $\eta_1\wedge\D\eta_1=\s\{\eta_1\otimes\D\eta_1\}$ holds and because of \eqref{Fa-prop}, \eqref{Fetaxia} and the fact that $\xi_1$ is Killing, it is valid the following
\begin{equation}\label{1}
(\eta_1\wedge\D\eta_1)(x,y,z)=-2\sx\{\eta_1(x)F_1(y,\f_1 z,\xi_1)\}.
\end{equation}
Moreover, from the equalities $\D \Phi(x,y,z)=-\sx F_1(x,y,z)$ and \eqref{Fa-prop}, we get
\begin{equation}\label{2}
\D^{\f_1} \Phi(x,y,z)=-\sx\{F_1(\f_1 x,y,z)+2F_1(x,\f_1 y,\xi_1)\eta_1(z)\}.
\end{equation}
So, applying \eqref{1}, \eqref{2}, \eqref{N1=F1} and \eqref{Fa-prop} to the equality \eqref{T1Kuch},
we obtain an expression of $T_1$ in terms of $F_1$ as follows
\begin{equation}\label{T1}
\begin{array}{l}
T_1(x,y,z)=F_1(x,y,\f_1 z)-F_1(y,x,\f_1 z)-F_1(\f_1 z,x,y)\\
\phantom{T_1(x,y,z)=}
+2F_1(x,\f_1 y,\xi_1)\eta_1(z).
\end{array}
\end{equation}

The equivalences in the following theorem are known from \cite{Man31} and \cite{ManIv36}.
\begin{thm}\label{thm:F3F7}
  The following statements for an almost contact B-metric manifold $(\MM,\f_2,\xi_2,\eta_2,g)$ are equivalent:
 \begin{enumerate}
 \item
 It belongs to the class $\F_3\oplus\F_7$, which is characterised by the conditions: the cyclic sum of $F_2$ by the three arguments vanishes and $\xi_2$ is Killing;
 \item
 It has a vanishing associated {Nijenhuis} tensor $\widehat N_2$;
 \item
 It has a vanishing tensor $\{\f_2,\f_2\}$ and $\xi_2$ is Killing;
 \item
 It admits the existence of a unique natural connection $D^2$ with totally skew-sym\-met\-ric torsion.
\end{enumerate}
\end{thm}
\begin{proof}
The equivalence of (i), (ii) and (iv) is known from \cite{Man31} and \cite{ManIv36}, whereas
the equivalence of (ii) and (iii) follows from \eqref{hatN2} and \propref{prop:hatN=0=>Kill}.
\end{proof}

For the natural connection $D^2$ with totally skew-symmetric torsion for the structure $(\f_2,\xi_2,\eta_2,g)$, we have
\begin{equation}\label{D2T2}
g\left(D^2_{x} y,z\right)=g(\n_x y, z)+\frac12 T_2(x,y,z),
\end{equation}
where its torsion $T_2$ is determined by
$
T_2=\eta_2\wedge \D\eta_2+\frac{1}{4}\s N_2
$
and it is expressed
in terms of $F_2$ by
\begin{equation}\label{T37} %
T_2(x,y,z)=-\frac{1}{2} \sx\bigl\{F_2(x,y,\f_2 z)
-3\eta_2(x)F_2(y,\f_2 z,\xi_2)\bigr\}. %
\end{equation} %

%*******************************************************
%
%Verojatno za syshtestvuvaneto na svyrzanostta trjabva da imame $\{\fa,\fa\}=0$, i.e. asoc. Nijenh. vyrhu $\ker(\etaa)=0$ trjabva da e nulev, t.e. $\fa$ da e pochti kompleksna struktura s nulev asoc. Nijenhuis tensor!
%
%
%*******************************************************

Using \thmref{thm:NN=Nhat}, \thmref{thm:F3F7} and \propref{prop:hatN=0=>Kill}, we get immediately the following
\begin{thm}\label{thm:nat-conn}
    Let $(\MM,\fa,\xia,\etaa,g)$, $(\al=1,2,3)$ be a manifold with almost contact HN-metric 3-structure.
    The existence of unique natural connections with totally skew-sym\-met\-ric torsion for two of the three structures implies an existence of a unique natural connection with totally skew-symmetric torsion for the remaining third structure.
\end{thm}

\begin{cor}\label{cor-G1}
Let $(\MM,\fa,\xia,\etaa,g)$, $(\al=1,2,3)$ be a manifold with almost contact HN-metric 3-structure.
If the manifold belongs to two of the following three classes for the corresponding structure, then the manifold belongs to the remaining third class for the corresponding structure:
%of the following classes for the other structure:
  \begin{enumerate}
    \item $\W_{2}\oplus\W_{4}\oplus\W_{9}\oplus\W_{10}\oplus\W_{11}$ for $\al=1$;
    \item $\F_3\oplus\F_7$  for $\al=2$;
    \item $\F_3\oplus\F_7$  for $\al=3$.
  \end{enumerate}
\end{cor}

%\subsection{United natural connection with totally skew-symmetric torsion}

%\textbf{Da se dopylnjat teoremite kato pri hypercomplex: koga syshtestvuva edinstvena $\f_\al$KT-sonnection
%za vsjako $\f_\al$ i koga tezi tri svyrzanosti syvpadat.  }

Now, we are interested on conditions for coincidence of these three natural connections $D^\al$, $(\al=1,2,3)$ with totally skew-symmetric torsion for the particular almost contact structures with the metric $g$. Then we shall say that it exists a \emph{natural connection with totally skew-symmetric torsion for the almost contact HN-metric 3-structure}.

\begin{thm}\label{thm:D}
Let $(\MM,\fa,\xia,\etaa,g)$, $(\al=1,2,3)$ be a manifold with almost contact HN-metric 3-structure for which the associated Nijenhuis tensors $\widehat N_\al$ vanish and $\xi_1$ is Killing.
This manifold admits
a linear connection $D$ with totally skew-sym\-met\-ric torsion preserving the almost contact HN-metric 3-structure if and only if the following equalities are valid
\begin{equation}\label{*}
\begin{array}{l}
F_1(x,y,\f_1 z)-F_1(y,x,\f_1 z)-F_1(\f_1 z,x,y)%\\
%\phantom{T_1(x,y,z)=}
+2F_1(x,\f_1 y,\xi_1)\eta_1(z)
\\
=-\frac{1}{2} \sx\bigl\{F_2(x,y,\f_2 z)
-3\eta_2(x)F_2(y,\f_2 z,\xi_2)\bigr\} %
\\
=-\frac{1}{2} \sx\bigl\{F_2(x,y,\f_3 z)
-3\eta_3(x)F_3(y,\f_3 z,\xi_3)\bigr\}. %
\end{array}
\end{equation}
If $D$ exists, it is unique.
\end{thm}
\begin{proof}
According to \thmref{thm:NN=Nhat} and \thmref{thm:F3F7}, since $\widehat N_\al=\LLL_{\xi_1}g=0$ are valid then
there exist the natural connections $D^\al$, $(\al=1,2,3)$ with totally skew-sym\-metric torsion $T_\al$ for the structures $(\fa,\xia,\etaa,g)$.
Bearing in mind \eqref{D1T1}, \eqref{T1}, \eqref{D2T2} and \eqref{T37}, the coincidence of $D^1$, $D^2$ and $D^3$ is equivalent to the conditions \eqref{*}. %This completes the proof.
\end{proof}

\section{A 7-dimensional Lie group as a manifold with almost contact HN-metric 3-structure}

%In this subsection we construct a 7-dimensional manifold with an almost contact 3-structure and metric of Hermitian-Norden type.

Let $\LL$ be a 7-dimensional real connected Lie group, and
$\mathfrak{l}$ be its Lie algebra with a basis
$\{e_1,e_2,e_3,e_4,e_5,e_6,e_7\}$.
Then an arbitrary vector $x$ in $T_p\LL$ at $p\in \LL$ is presented by $x=x^{i}e_i$ $(i=1,2,\dots,7)$.

Now we introduce an almost contact HN-metric 3-structure
$(\f_{\al},\xi_{\al},\eta_{\al})$ by a standard way as follows %in \eqref{f-baza} for $n=1$:
\begin{equation}\label{str-exa}
\begin{array}{l}
\begin{array}{llll}
\f_1e_1=e_2,\quad & \f_1e_2=-e_1,\quad & \f_1e_3=e_4, \quad & \f_1e_4=-e_3,\\
    \f_1e_5=o,\quad & \f_1e_6=e_7, \quad & \f_1e_7=-e_6,\\
\f_2e_1=e_3, \quad & \f_2e_2=-e_4, \quad & \f_2e_3=-e_1, \quad & \f_2e_4=e_2,\\
    \f_2e_5=-e_7,\quad & \f_2e_6=o, \quad & \f_2e_7=e_5,\\
\f_3e_1=e_4,\quad & \f_3e_2=e_3, \quad & \f_3e_3=-e_2, \quad & \f_3e_4=-e_1, \\
    \f_3e_5=e_6, \quad & \f_3e_6=-e_5, \quad & \f_3e_7=o,\\
\xi_1=e_5,\quad & \xi_2=e_6, \quad & \xi_3=e_7,\quad \\
\eta_1=x^5, \quad & \eta_2=x^6, \quad & \eta_3=x^7,
\end{array}
\end{array}
\end{equation}
where $o$ is the zero vector in $T_p\LL$, $p\in \LL$.

Let $g$ be a pseudo-Riemannian metric such that
\[
\begin{array}{l}
  g(e_1,e_1)=g(e_2,e_2)=-g(e_3,e_3)=-g(e_4,e_4)\\
  \phantom{g(e_1,e_1)}
            =-g(e_5,e_5)=g(e_6,e_6)=g(e_7,e_7)=1, \\
  g(e_i,e_j)=0,\; i\neq j.
\end{array}
\]

The almost contact HN-metric 3-structure
$(\f_{\al},\xi_{\al},\eta_{\al})$ on $H$ coincides with the almost hypercomplex HN-metric structure considered in \cite{GriManDim12}. The almost hypercomplex structure is defined as in \cite{Som}.

Let us consider $(\LL,\f_{\al},\xi_{\al},\eta_{\al},g)$ with the Lie algebra $\mathfrak{l}$
determined by the following nonzero commutators:
\[
\left
[e_1,e_2\right]=[e_3,e_4]=\lm e_7,
\]
where $\lm\in\R\setminus\{0\}$.

%In terms of the almost contact 3-structure, if $e_1$ be denoted briefly as $e$, \eqref{lie} can be written as
%\begin{equation}\label{lie=}
%\left
%[e,\f_1e\right]=[\f_2e,\f_3e]=\lm \xi_3,\qquad \lm\neq 0.
%\end{equation}

By the well-known Koszul equality, we compute the components of the Levi-Civita connection $\n$ with respect to the basis and the nonzero ones of them are: %covariant derivatives of the basis vectors %by the Koszul equality
\begin{equation}\label{nabli-exa}
\begin{array}{c}
\begin{array}{c}
\n_{e_1} e_2=-\n_{e_2} e_1=\n_{e_3}e_4=-\n_{e_4} e_3=\frac12\lm e_7,
\end{array}\\
\begin{array}{ll}
\n_{e_1} e_7=\n_{e_7} e_1=-\frac12\lm  e_2, &\quad
\n_{e_2} e_7=\n_{e_7} e_2=\frac12\lm  e_1, \\
\n_{e_3} e_7=\n_{e_7} e_3=\frac12\lm  e_4, &\quad
\n_{e_4} e_7=\n_{e_7} e_4=-\frac12\lm  e_3.
\end{array}
\end{array}
\end{equation}

\begin{prop}\label{prop:exa-class}
Let $(\LL,\f_{\al},\xi_{\al},\eta_{\al},g),$ $(\al=1,2,3),$ be the Lie group $\LL$ with almost contact HN-metric 3-structure depending on the nonzero real parameter $\lm$. Then this manifold belongs to the following basic
classes, according to the corresponding classification in \emph{\cite{AlGa}} and \emph{\cite{GaMiGr}}$:$
\begin{itemize}
  \item $\W_{10}$ with respect to $(\f_{1},\xi_{1},\eta_{1},g);$
  \item $\F_{3}$ with respect to $(\f_{2},\xi_{2},\eta_{2},g);$
  \item $\F_{7}$ with respect to $(\f_{3},\xi_{3},\eta_{3},g).$
\end{itemize}
\end{prop}
\begin{proof}
Using \eqref{F}, \eqref{str-exa} and \eqref{nabli-exa}, we obtain the basic components of tensors $(F_{\al})_{ijk}\allowbreak{}=F_{\al}(e_i,e_j,e_k)$ as follows:
\begin{equation}\label{F1F2F3}
\begin{array}{l}
\frac12\lm  =(F_1)_{117}=(F_1)_{126}=-(F_1)_{216}=(F_1)_{227}\\
\phantom{\frac12\lm}
            =(F_1)_{337}=(F_1)_{346}=-(F_1)_{436}=(F_1)_{447}\\
\phantom{\frac12\lm}
            =(F_2)_{125}=(F_2)_{147}=-(F_2)_{215}=(F_2)_{237}\\
\phantom{\frac12\lm}
            =-(F_2)_{327}=(F_2)_{345}=-(F_2)_{417}=-(F_2)_{435}\\
\phantom{\frac12\lm}
            =-(F_3)_{137}=(F_3)_{247}=(F_3)_{317}=-(F_3)_{427}.
\end{array}
\end{equation}

%As a sequel we have that the Lee forms $\theta_{\al}$, $\theta^*_{\al}$, $\omega_{\al}$  $(\al=1,2,3)$ vanish.

From \eqref{F1F2F3}, applying the classification conditions for the relevant classification in \cite{AlGa} or \cite{GaMiGr}, we have the classes in the statement, respectively.
\end{proof}

Bearing in mind \propref{prop:exa-class}, we deduce that the conditions (i) of \thmref{thm:NN=Nhat} and \thmref{thm:F3F7} are fulfilled and therefore there exist natural connections $D^\al$ ($\al=1,2,3$) for the corresponding structure $(\f_\al,\xi_\al,\eta_\al,g)$ on $\LL$.
We get the components with respect to the basis of their torsions $T_\al$ ($\al=1,2,3$) by direct computations from \eqref{D1T1}, \eqref{T1}, \eqref{D2T2}, \eqref{T37} and \eqref{F1F2F3} as follows
\[
\begin{array}{c}
(T_1)_{127}=(T_1)_{347}=-\lm,
\\
(T_2)_{127}=-(T_2)_{145}=-(T_2)_{235}=(T_2)_{347}=-\frac12 \lm,
\\
(T_3)_{127}=(T_3)_{347}=-\lm.
\end{array}
\]
Obviously, the connections $D^1$ and $D^3$ coincides but $D^2$ differs from them.
The condition \eqref{*} of \thmref{thm:D} is not fulfilled and therefore it does not exist a unique connection with totally skew-sym\-met\-ric torsion preserving the almost contact HN-metric 3-structure on $\LL$.

%By virtue of \eqref{nabli-exa}, the nonzero components of the curvature tensor are:
%\begin{equation}\label{R-exa}
%\begin{array}{l}
%R_{1221}=R_{3443}=-\frac34\lm ^2,\qquad
%R_{1234}=\frac12\lm ^2,\\
%R_{1771}=R_{2772}=-R_{3773}=-R_{4774}=-R_{1342}=R_{1432}=\frac14\lm ^2.
%\end{array}
%\end{equation}
%and others are obtained by the symmetries of $R$.
%
%From here we obtain the nonzero basic components of the Ricci tensor and the values of the scalar curvature:
%\begin{equation}\label{ro-exa}
%\begin{array}{c}
%\rho_{11}=\rho_{22}=-\rho_{33}=-\rho_{44}=-\frac12\lm ^2,\quad \rho_{77}=\lm ^2;
%\end{array}
%\end{equation}
%%
%\begin{equation}\label{tau-exa}
%\begin{array}{c}
%\tau=-\lm ^2,\quad
%\tau^*_1=-2\tau,\quad %\tau^*_2=-2\lm 0,\quad  \tau^*_3=200,\quad  &
%\tau^{**}_2=\tau^{**}_3=-3\tau.
%\end{array}
%\end{equation}
%
%
%Bearing in mind \eqref{R-exa}, \eqref{ro-exa} and \eqref{tau-exa}, we have
%\begin{prop}\label{prop:Ein}
%The manifold $\Lfa$ is
%a fiber bundle over a $4$-dimen\-sion\-al almost hypercomplex HN-metric manifold that is Einstein. %???
%%\begin{enumerate}
%%  \item if and only if $\lm\neq 0$;
%%  \item flat if and only if $\lm = 0$.
%%\end{enumerate}
%\end{prop}
%
%%
%%{The natural connection with totally skew-symmetric torsion?
%%}
%
%
%
%
%
%
%
%
%{The natural connection with totally skew-symmetric torsion?
%}
%

%\subsection*{Acknowledgement}
This work was partially supported by project NI15-FMI-004 of the
Scientific Research Fund at the University of Plovdiv.
The author wishes to thank Professor Stefan Ivanov for valuable discussions about the present paper.
%\end{acknowledgement}

%\bibliographystyle{numbers}
%\bibliographystyle{abbrv}

%\bibliography{...}

% ------------------------------------
\end{document}